\documentclass[11pt,a4paper]{amsart}
\usepackage{amsthm,amsmath,amscd}
\title{Minimal free resolutions for homogeneous ideals with Betti numbers $1,n,n,1$}
\author{Alfio Ragusa
	\and  Giuseppe Zappal\`a}

\subjclass[2010]{13 H 10, 14 N 20, 13 D 40}
\keywords{Graded minimal free resolutions, Graded Betti numbers, Structure Theorems, Homological dimension, Gorenstein ideals}

\DeclareSymbolFont{rsfscript}{OMS}{rsfs}{m}{n}
\DeclareSymbolFontAlphabet{\mathrsfs}{rsfscript}

\DeclareSymbolFont{AMSb}{U}{msb}{m}{n}
\DeclareSymbolFontAlphabet{\mathbb}{AMSb}

\DeclareSymbolFont{eufrak}{U}{euf}{m}{n}
\DeclareSymbolFontAlphabet{\gothic}{eufrak}

\def\ac{\`}

\newcommand{\f}{\footnotesize}

\newcommand\GCD{\operatorname{GCD}}

\newcommand{\pp}{\mathbb P}
\newcommand\hgt{\operatorname{ht}}
\newcommand\depth{\operatorname{depth}}

\newcommand\rank{\operatorname{rank}}

\newcommand\imm{\operatorname{Im}}

\newcommand\rw{\Rightarrow}

\newcommand\codim{\operatorname{codim}}

\newcommand\Imm{\operatorname{Im}}

\newcommand\zz{{\mathbb Z}}

\newcommand{\xr}{\xrightarrow}
\newcommand{\bo}{\bigoplus}

\DeclareMathOperator{\Coker}{Coker}
\DeclareMathOperator{\Ker}{Ker}

\newtheorem{thm}{Theorem}[section]
\newtheorem{lem}[thm]{Lemma}
\newtheorem{prp}[thm]{Proposition}
\newtheorem{cor}[thm]{Corollary}

\theoremstyle{definition}

\newtheorem{dfn}[thm]{Definition}

\theoremstyle{remark}
\newtheorem{rem}[thm]{Remark}
\newtheorem{exm}[thm]{Example}

\newcounter{num}

\begin{document}



\begin{abstract}
We investigate the standard generalized Gorenstein algebras of homological dimension three, giving a structure theorem for their resolutions. Moreover in many cases we are able to give a complete description of their graded Betti numbers.
\end{abstract}

\maketitle

\section*{Introduction}
\markboth{\it Introduction}{\it Introduction}

One of the most important tool for studying projective schemes is the minimal free resolution of their defining ideals. When the scheme has homological dimension $2,$ there is the Hilbert-Burch structure theorem which allows a deep knowledge of these schemes
(see \cite{B}). 
When we study schemes having homological dimension $3,$ the matter becomes a little intriguing even in the simplest case of the arithmetically Cohen Macaulay schemes. In this last case, when the rank of the last syzygy module 
is $1,$ the situation becomes analogous to the homological dimension $2$ case according to the structure theorem by Buchsbaum and Eisenbud (see \cite {BE}). These are the arithmetically Gorenstein schemes of codimension $3.$ 
\par
So it seems completely natural to study all the schemes of homological dimension $3$ such that the rank of the last syzygy module 
is $1,$ even for codimension less than $3.$ The goal of this paper is just to study such schemes, which we will call {\em generalized Gorenstein schemes} (of course we omit the trivial case when the codimension is $1$),  
\par
These schemes arise frequently in different contexts. For instance the disjoint union of two complete intersection curves in $\pp^3$ and some subschemes of star configurations (see \cite{RZ3}) are generalized Gorenstein schemes.
As in the Gorenstein case, the central map of the resolution is represented by a square matrix $M$ of submaximal rank, which is named {\em presentation matrix}. This matrix, once more, completely determines the resolution of the defining ideals.
\par
Therefore, in section $1,$ after studying properties of presentation matrices $n\times m,$ $n\le m$, we give a first  characterization of such matrices in Proposition \ref{carm}.
\par
Section $2$ is dedicated to producing a structure theorem for the resolutions of the generalized Gorenstein algebras of homological dimension $3$ and the main result on this direction is Theorem \ref{pres}. Moreover we study the special case in which the ideal associated to the last map in the resolution is minimally generated by $3$ elements. In the Propositions \ref{cn} and \ref{cninv} we give a complete characterization of such schemes and in Theorem \ref{sch} we give a nice geometrical description of these schemes.
\par
All these results permit us to study the graded Betti numbers for generalized Gorenstein schemes of homological dimension $3.$ The main result in section $3$  is in Theorem \ref{treB}, in which all the graded Betti numbers are characterized in the case in which the defining ideal is minimally generated by $n=3$ elements. The main results of section $4$ are in Theorem \ref{omogdisp},  where we give a complete description of the graded Betti sequences  for those schemes which  have $n$ generators and syzygies with concentrated degrees, with $n$ odd,  and in Proposition \ref{omogparisuff}  where some necessary conditions and some sufficient conditions are given in the even case.


\section{matrices of submaximal rank}
\markboth{\it Square matrices of submaximal rank}{\it Square matrices of submaximal rank}
Throughout the paper $k$ will be a field and $R:=k[x_1,\ldots,x_r],$ $r\ge 3,$ will be the standard graded polynomial $k$-algebra.
\par
The aim of this paper is to investigate graded minimal free resolutions of type 
   $$0\to R(-s)\to \bigoplus_{j=1}^nR(-b_j) {\rightarrow} \bigoplus_{i=1}^nR(-a_i)\to R.$$
We give the following definition.
\begin{dfn}\label{defmin} 
Let $M\in R^{n,m}$ be a matrix with $m\ge n.$ We say that $M$ is a {\em presentation matrix} if it is associated to a map $\varphi$ in
a presentation of type
 $$R^m\xrightarrow{\varphi}R^n\to R.$$
\end{dfn}
Note that if $M\in R^{n,m}$ is a presentation matrix, then $\rank M=n-1.$
\par
At first in this section we would like to study properties of presentation matrices.
\par
Let $S$ be a $UFD$ and let $H\in S^{n,n-1}.$ We set $g_i(H)=(-1)^{i+1}H_i$ where $H_i$ is the minor of $H$ obtained by deleting the $i$-th row of $H.$
The following lemmas will play a key role along the paper.

\begin{lem}\label{pol}
Let $S$ be a $UFD,$ $M\in S^{n,m},$ $n\le m,$ with $\rank M=n-1.$ 
Let $N$ and $N'$ be two submatrices of $M$ of size $n\times (n-1)$ of rank $n-1.$ Let $d_N=\GCD(g_1(N),\ldots,g_n(N))$ and
$d_{N'}=\GCD(g_1(N'),\ldots,g_n(N')).$ Then there exists a unit $a\in S$ such that
 $$\left(\frac{g_1(N)}{d_N},\ldots,\frac{g_n(N)}{d_N}\right)=a
 \left(\frac{g_1(N')}{d_{N'}},\ldots,\frac{g_n(N')}{d_{N'}}\right).$$ 
\end{lem}
\begin{proof}
Let $Q(S)$ be the field of fractions of $S.$  Let us consider
 $$v=\left(\frac{g_1(N)}{d_N},\ldots,\frac{g_n(N)}{d_N}\right)\in Q(S)^n,\,\,v'=\left(\frac{g_1(N')}{d_{N'}},\ldots,\frac{g_n(N')}{d_{N'}}\right)\in Q(S)^n,$$ 
then we have that $\langle v\rangle=\langle v'\rangle,$ as vector subspaces of $Q(S)^n,$ since $\rank M=n-1.$ Therefore $bv=b'v',$ with $b,b'\in S,$ $\GCD(b,b')=1,$ so $b$ is a divisor of each component of the vector $v',$ hence $b$ is unit. Analogously $b'$ is a unit too, so $v=av',$ with $a$ a unit.
\end{proof}

By Lemma \ref{pol}, using the same notation, we can set $g_i(M)=g_i(N)/d_N$ and we will set 
 $$\gamma(M)=(g_1(M),\ldots,g_n(M)).$$
Note that $\gamma(M)$ is determined up to a unit and $\GCD(\gamma(M))=1.$ So $\gamma(M)$ generates an ideal $I_M,$ with $\depth I_M\ge 2.$ In the sequel $I_M$ will be called {\em the ideal associated} to $M.$ Moreover $\gamma(M)$ defines a map
  $$\gamma(M):R^n\to R.$$
When $F$ and $G$ are graded free modules, $\rank F\ge\rank G$ and
$\varphi:F\to G,$ is a map associated to a matrix $M$ (with respect to some bases), with $\rank\varphi=\rank G-1,$
we set $\gamma(\varphi)=\gamma(M).$ So $\gamma(\varphi):G\to R.$ By construction we have
 $$\gamma(\varphi)\varphi=0.$$
\begin{cor}\label{sist}
With the same notation of Lemma \ref{pol}, if $(h_1, \ldots, h_n)M=0$ then $(h_1, \ldots, h_n)=\lambda \gamma(M)$ for some $\lambda\in S.$
\end{cor}
\begin{proof} As in the proof of Lemma \ref{pol} $b (h_1, \ldots, h_n)=a \gamma(M)$ with $a,b \in S.$ Now since $\GCD(\gamma(M))=1$ we  see that $b$ divides $a$, so the conclusion follows.
\end{proof}

The following proposition deals with the special case when $I_M$ is perfect of height $2.$

\begin{prp}\label{gcd}
Let $S$ be a $UFD$ and $M\in S^{n,m},$ $n\le m,$ with $\rank M=n-1.$ Let us suppose that a submatrix $M'$ obtained taking $n-1$ columns of $M$ has rank $n-1$ and its maximal minors are coprime. Then $I_M$ is a perfect ideal of height two and each column of $M$ is in the module generated by the columns of $M'.$
\end{prp}
\begin{proof}
By the hypothesis, $M'$ is a Hilbert-Burch matrix of $I_M=(g_1,\ldots,g_n),$ where the $g_i$'s are the maximal minors of $M'.$ So $I_M$ is perfect ideal of height two and the columns of $M'$ generate the syzygy module on $(g_1,\ldots,g_n).$ Since any column of $M$ is a syzygy on $(g_1,\ldots,g_n),$ the conclusion follows.
\end{proof}

\begin{rem}
In particular, when $M\in R^{n,m},$ $n\le m,$ is a presentation matrix with a submatrix $M'$ as in Proposition \ref{gcd}, then the presentation defined by $M$ is not minimal.
\end{rem}

\begin{prp}\label{carm}
Let $M\in R^{n,m},$ $m\ge n$ with $\rank M=n-1.$ Let $C(M)$ be the module generated by the columns of $M.$
\par
$M$ is a presentation matrix iff every syzygy on $\gamma(M)$ belongs to the module $C(M).$
\end{prp}
\begin{proof}
If $M$ is a presentation matrix then there exists an exact complex
 \begin{equation}\label{mfr} R^m\stackrel{\varphi}{\rightarrow}R^n\stackrel{\psi}{\rightarrow}R,\end{equation}
where $\varphi$ is the map associated to $M.$ 
Let $(f_1, \ldots, f_n)=\Imm\psi.$ Then $$(f_1 \ldots f_n)M=0$$ and by Corollary \ref{sist} $(f_1, \ldots, f_n)=\lambda \gamma(M),$ for some $\lambda \in R.$ Thus, every syzygy on $\gamma(M)$ is a syzygy on $(f_1, \ldots, f_n)$ too. Since the complex (\ref{mfr}) is exact, we are done. 
\par
Vice versa, let us consider the following complex
\begin{equation}\label{mfr'}R^m\xrightarrow{\varphi}R^n\xrightarrow{\gamma(\varphi)}R,\end{equation}
where $\varphi$ is the map associated to $M.$
Since, by the hypothesis, $\Ker\gamma(\varphi)\subseteq\imm \varphi,$ $M$ is a presentation matrix. 
\end{proof}

\begin{exm}
Let $R=k[x,y,z,t]$ and let 
 $$M=\begin{pmatrix}y & -x & 0 & 0 \\ 0 & z & -y & 0 \\ 0 & 0 & t & -z\\-t & 0 & 0 & x\end{pmatrix}.$$
We have that $\gamma(M)=(zt,xt,xy,yz)=(x,z)\cap (y,t),$ so it is easy to verify that $M$ is a presentation matrix, whereas $M^T$ is not. In fact 
$\gamma(M^T)=(x,y,z,t),$ so it does not satisfy the hypothesis of Proposition \ref{carm}.
\par

\end{exm}


\section{The structure of the resolution}
\markboth{\it The structure of the resolution}{\it The structure of the resolution}

In this section we will deal with the case in which $M$ is a presentation square matrix of size $n\ge 3.$
Now we define the class of algebras which we will investigate along this paper.

\begin{dfn}\label{ggor} 
A graded standard $R$-algebra $R/I$ of homological dimension $3$ is called {\em generalized Gorenstein algebra} if the rank of the last syzygy module is $1$ in a its minimal graded free resolution.
\end{dfn}

\begin{thm}\label{ris} 
Let $M\in R^{n,n}$ be a presentation square matrix. Then $R/I_M$ is a generalized Gorenstein algebra whose a free resolution is
 \begin{equation}\label{ris1} 0\to R^*\xrightarrow{\gamma(\varphi^*)^*}R^n\xrightarrow{\varphi}R^n 
\xrightarrow{\gamma(\varphi)}R\to R/I_M\to 0,\end{equation}
where $\varphi$ is the map associated to the matrix $M.$
\end{thm}
\begin{proof} 
Since $M$ is a presentation matrix, using Proposition \ref{carm}, we get $\Imm\varphi=\Ker\gamma(\varphi).$
Furthermore $\gamma(\varphi^*)\varphi^*=0$ implies that $\varphi\gamma(\varphi^*)^*=0.$
Now let $u\in\Ker\varphi.$ It defines a map $f_u:R\to R^n.$ Therefore we have that $\varphi f_u=0,$ hence $f_u^*\varphi^*=0.$
Using Corollary \ref{sist} we deduce that $f_u^*=a\gamma(\varphi^*),$ with $a\in R,$ consequently $f_u=a\gamma(\varphi^*)^*,$
i.e. $u\in\Imm\gamma(\varphi^*)^*.$
\par
Finally, by Lemma 2.6 in \cite{RZ2}, since $\rank\varphi=n-1,$ $\Ker\varphi$ is a free module of rank $1.$ 
\end{proof}

\begin{dfn}
Whenever the resolution of Theorem \ref{ris} is minimal, we say that $M$ is a {\em minimal presentation matrix}.
\end{dfn}

\begin{rem} If $M\in R^{n,n}$ is a presentation matrix then $\hgt I_M $ is either $2$ or $3.$ When $\hgt I_M =3$, $I_M$ is a Gorenstein ideal. Consequently if $n$ is even, then $\hgt I_M=2.$
\end{rem}

\begin{rem} If $M\in R^{n,n}$ is a presentation matrix, by the exactness of the complex (\ref{ris1}), then $\depth(\gamma(M^T))\ge 3$ and $\depth(I_{n-1}(M))\ge 2,$ by the exactness criterion in \cite{BE2}.
\end{rem}

Using The Theorem \ref{ris} we are able to characterize all the standard generalized Gorenstein algebras 
of homological dimension $3.$

\begin{cor}
Let $I$ be an ideal of $R.$ Then $R/I$ is a generalized Gorenstein algebra of homological dimension $3$ iff there exists a presentation matrix $M$ such that $I=I_M.$
\end{cor}
\begin{proof}
Theorem \ref{ris} says that every $I_M$ associated to a presentation matrix is a generalized Gorenstein ideal. Conversely, if $I$ is a 
generalized Gorenstein ideal in $R$ of homological dimension $3$ then there is a resolution of $R/I$ of type 
$$0\to R\to R^n\xr{M}R^n\to R \to R/I \to 0.$$
So $M$ is a presentation matrix hence, again by Theorem \ref{ris}, $\Coker (M) =I_M.$ From which we get $I=I_M.$
\end{proof}

Thus it is important to recognize when $M$ is a presentation matrix. The next results will give another characterization of such matrices.

\begin{lem}\label{lcof}
Let $M\in R^{n,n}$ be a matrix of rank $n-1.$ We set $\gamma(M)=(g_1,\ldots,g_n)$ and $\gamma(M^T)=(h_1,\ldots,h_n).$
Let $M^C$ be the cofactor matrix of $M.$
Then $M^C=u(g_ih_j),$ where $u\in R.$
\end{lem}
\begin{proof}
Let us consider the following complex
 $$0\to R\xr{(h_1\ldots h_n)^T}R^n\xr{M}R^n\xr{(g_1\ldots g_n)}R.$$
We denote by $c_{ij}$ the cofactor of $M$ in the position $(i,j).$
Since $(g_1\ldots g_n)M=0$ and $(c_{1j}\ldots c_{nj})M=0,$ by Corollary \ref{sist} we get
 $$(c_{1j},\ldots, c_{nj})=\lambda_j(g_1,\ldots, g_n).$$
On the other hand since $(h_1\ldots h_n)M^T=0$ and $(c_{i1}\ldots c_{in})M=0,$ again by Corollary \ref{sist} we get
 $$(c_{i1},\ldots, c_{in})=\mu_i(h_1,\ldots, h_n).$$
By these equalities follows that
 $$\lambda_j(g_1,\ldots g_n)=h_j(\mu_1,\ldots,\mu_n);\,\,\mu_i(h_1,\ldots, h_n)=g_i(\lambda_1,\ldots,\lambda_n).$$
Since $\GCD(g_1,\ldots g_n)=\GCD(h_1,\ldots h_n)=1,$ we obtain that $$c_{ij}=\alpha_jh_jg_i=\beta_ig_ih_j,$$ for some 
$\alpha_j,\beta_i\in R,$ so $\alpha_j=\beta_i$ for every $i$ and $j.$ Let 
 $$u=\alpha_1=\ldots=\alpha_n=\beta_1=\ldots=\beta_n,$$ thus $c_{ij}=ug_ih_j.$
\end{proof}

\begin{thm}\label{pres}
Let $M\in R^{n,n}$ be a matrix of rank $n-1.$ Let $\gamma(M)=(g_1,\ldots,g_n),$ $\gamma(M^T)\!=\!(h_1,\ldots,h_n)$ and let
$J$ be the ideal generated by $\gamma(M^T).$ Let $M^C$ be the cofactor matrix of $M.$
\par
The matrix $M$ is a presentation matrix iff $\depth J\ge 3$ and $M^C=u(g_ih_j)$ where $u$ is a unit.
\end{thm}
\begin{proof}
We set $I=(g_1,\ldots,g_n).$
If $M$ is a presentation matrix then by Theorem \ref{ris} the complex
 \begin{equation}\label{Mcompl} 0\to R\xr{(h_1\ldots h_n)^T}R^n\xr{M}R^n\xr{(g_1\ldots g_n)}R \end{equation}
is exact, so by the Buchsbaum-Eisenbud criterion, we get $\depth J\ge 3$ and $\depth I_{n-1}(M)\ge 2.$ Note that $I_{n-1}(M)$ is generated by the entries of $M^C,$ so by Lemma \ref{lcof}, $I_{n-1}(M)=uIJ,$ consequently $u$ is a unit. 
\par
Conversely let us suppose that $\depth J\ge 3$ and $M^C=u(g_ih_j)$ where $u$ is a unit. To show that $M$ is a presentation matrix it is enough to show that the complex (\ref{Mcompl}) is exact. To do this, we will use the Buchsbaum-Eisenbud criterion.
The conditions about the ranks of the modules are trivially satisfied. We have that $\depth I\ge 2,$ by definition of $\gamma(M),$ and 
$\depth J\ge 3$ by the hypothesis. It remains only to prove that $\depth I_{n-1}(M)\ge 2.$ Since $M^C=u(g_ih_j)$ where $u$ is a unit, we have that $I_{n-1}(M)=IJ.$ Let $f\in IJ,$ $f\ne 0.$ Since $\depth I\ge 2$ and 
$\depth J\ge 3,$ there exist $g\in I$ and $h\in J$ such that $(f,g)$ and $(f,h)$ are regular sequences. Hence $(f,gh)$ is a regular sequence in $IJ$ and we are done.
\end{proof}

\begin{rem}\label{pfaff}
Let $M$ be an alternating matrix of odd size $n$ and of rank $n-1.$ 
Then $\gamma(M)=\gamma(M^T)$ and $M^C=(p_ip_j),$ where $p_h$ is the $h$-th submaximal pfaffian of $M.$
Let $P_M$ be the ideal generated by the submaximal pfaffians of $M.$ Suppose that $\depth P_M=3.$ Then $M$ is a presentation matrix and $P_M=I_M.$ So Theorem \ref{pres} allows us to recover the well known characterization of the Gorenstein ideals of heighth $3$ of Buchsbaum and Eisenbud \cite{BE}.
\end{rem}

\begin{rem}\label{zeroh}
By Theorem \ref{pres}, let $M$ be a presentation matrix and let $(h_1,\ldots,h_n)=\gamma(M^T).$ Then $h_j=0$ iff the submatrix obtained from $M$ by removing the $j$-th column has not maximal rank.
\end{rem}

The following propositions put into relation a presentation square matrix $M$ with the kernel of the associated map.
\begin{prp}\label{dual}
Let $M\in R^{n,n}$ be a presentation square matrix and let $\varphi:R^n\to R^n$ be the associated map. 
Let $J=\Imm\gamma(\varphi^*).$ Let us suppose that
 $$R^m\xrightarrow{\psi} R^{n*}\xrightarrow{\gamma(\varphi^*)} R$$
is a presentation of $J.$ Then there exists a map $\beta:R^{m*}\to R^n,$ such that $\varphi=\beta\psi^*.$
\end{prp}
\begin{proof}
At first we dualize the resolution of Theorem \ref{ris}. We get the complex
 $$0\to R^*\xrightarrow{\gamma(\varphi)^*}R^{n*}\xrightarrow{\varphi^*}R^{n*}
\xrightarrow{\gamma(\varphi^*)}R.$$
By the exactness of the presentation we get the factorization $\varphi^*=\psi\alpha$ for a suitable $\alpha:R^{n*}\to R^m.$
Consequently we have $\varphi=\alpha^*\psi^*,$ so $\alpha^*$ is the required map $\beta.$
\end{proof}

\begin{prp}\label{dual-2}
Let $M\in R^{n,n}$ be a presentation square matrix and let $\varphi:R^n\to R^n$ be the associated map. 
Let us suppose that $\gamma(M^T)=(h_1,\ldots,h_t,0,\ldots,0).$ Let $J=(h_1,\ldots,h_t).$ Take a presentation of $J$
 \begin{equation}\label{pres-2} R^m\xrightarrow{\psi} R^{t^*}\xrightarrow{\tau} R \end{equation}
where $\tau(e_i^*)=h_i,$ for $1\le i\le t.$
Then there exist two maps $\beta:R^{m*}\to R^n$ and $\delta:R^{n-t}\to R^n$ such that 
$\varphi=(\beta\psi^*)\oplus\delta.$
\end{prp}
\begin{proof}
By the hypotheses $R/I_M$ has a resolution of the type
 $$0\to R^*\xr{(\tau^*,0)}R^{t}\oplus R^{n-t}\xr{\varphi=\varphi_1\oplus\delta}R^n\to R\to R/I_M\to 0.$$
Consequently we have $\varphi_1\tau^*=0,$ so we get the complex
 $$R^{n*}\xr{\varphi_1^*}R^{t^*}\xr{\tau} R.$$
By the exactness of (\ref{pres-2}), we get the factorization $\varphi_1^*=\psi\alpha$ 
for a suitable $\alpha:R^{n*}\to R^m,$ so $\varphi_1=\alpha^*\psi^*$ and $\varphi=(\alpha^*\psi^*)\oplus\delta.$

\end{proof}


The next results will be useful for studying generalized Gorenstein algebras.

\begin{lem}\label{zeri}
Let $M$ be a minimal presentation matrix and let $J$ be the ideal generated by $\gamma(M^T).$ Then there exists a resolution of $I_M$ of the type
 $$0\to R\xr{\rho}R^n\to R^n\to R\to R/I_M\to 0$$
such that $\rho(1)=(h_1,\ldots,h_s,0,\ldots,0)$ where $h_1,\ldots,h_s$ minimally generate $J.$
\end{lem}
\begin{proof}
Let $(h_1,\ldots,h_n)=\gamma(M^T).$ 
Let us suppose that $h_n=\sum_{i=1}^{n-1}a_ih_i.$ Let $\varphi:R^n\to R^n$ be the map associated to the matrix $M.$ We change the basis in the domain of $\varphi$ from $(e_1,\ldots,e_n)$ to $(v_1,\ldots,v_n),$ where $v_i=e_i+a_ie_n,$ for $1\le i\le n-1$ and $v_n=e_n.$ Then $\rho(1)=\sum_{i=1}^{n-1}h_iv_i.$
By iterating this procedure we get the stated result.
\end{proof} 

According to Lemma \ref{zeri} we will use the following notation. 
Let $I$ be a generalized Gorenstein ideal $I$ of homological dimension $3$ and let
 $$0\to R\xr{\rho}R^n\to R^n\to R\to R/I\to 0$$
be a minimal free resolution.
We define $\zeta(I)=\nu(I)-\nu(I(\rho)).$ 

Note that $0\le\zeta(I)\le\nu(I)-3.$
\par

Now we would like to study minimal free resolutions for generalized Gorenstein ideals $I$ of homological dimension $3$ with maximal $\zeta(I).$
Observe that in this case, using the same notation as before, $\rho(1)=(h_1,h_2,h_3,0,\ldots,0),$ 
where $(h_1,h_2,h_3)$ is a regular sequence.

\par
Using Proposition \ref{dual-2}, $R/I$ has a minimal free resolution of the type
 \begin{equation}\label{risol}0\to R\xr{(\tau,0)}R^{3}\oplus R^{n-3}\xr{(\alpha\kappa)\oplus\delta}R^n\to R,  \end{equation}
where $\tau(1)=(h_1,h_2,h_3),$ $\kappa:R^3\to R^3$ is the Koszul map on $h_1,h_2,h_3$ and $\alpha:R^3\to R^n,$ 
$\delta:R^{n-3}\to R^n$ are suitable maps.
Consequently $I=I_M,$ where $M$ is a minimal presentation matrix having the structure $M=(AK|C),$ with 
  $$K=\begin{pmatrix}0 & h_3 & -h_2\\ -h_3 &0 & h_1\\ h_2 & -h_1 & 0\end{pmatrix},$$
for some $A\in R^{n,3}$ and $C\in R^{n,n-3}.$
In the next result we will give the structure of the generators of such ideals.

\begin{prp}\label{cn}
Let $I=I_M$ be a generalized Gorenstein ideal of homological dimension $3$ with maximal $\zeta(I),$ where $M=(AK|C).$
Then $I$ is generated by the maximal minors obtained by deleting one by one the first $n$ rows of the $(n+1)\times n$-matrix
  $$B=\begin{pmatrix} \\ & A & & & C &\\ \\ h_1 &h_2 & h_3 & 0 &\ldots & 0 \end{pmatrix}.$$
\end{prp}
\begin{proof}
To compute a minimal set of generators for $I,$ for instance $\gamma(M),$ it is enough to compute the maximal minors of a submatrix of $M$ obtained by choosing a submatrix of $M$ of size $n\times (n-1)$ of rank $n-1$ (see Lemma \ref{pol}). Note that, since $\rank(AK)=2,$ to obtain such a submatrix, we are forced to remove one of the first three columns.
\par
Let $M_{(i;j)}$ be the minor of $M$ obtained by deleting the row $i$ and the column $j.$ The following computation will show that, for some $s,$
 $$M_{(i;j)}=(-1)^sh_j B_i,$$
where $B_i$ is the minor of $B$ obtained by deleting the row $i.$ Hence $I=(B_1,\ldots,B_n).$ 
\par
In fact we write the matrix $A$ by columns $A=(A_1 A_2 A_3)$ and $M$ in this way
 $$M=(-h_3A_2+h_2A_3\,|\,h_3A_1-h_1A_3\,|\,-h_2A_1+h_1A_2\,|\,C).$$
Moreover we will write $A_j^{(i)}$ the submatrix obtained by $A_j$ by removing the $i$-th row.
So 
$$M_{i,1}=|h_3A_1^{(i)}\,\,-h_2A_1^{(i)}\,\, C|+|h_3A_1^{(i)}\,\,h_1A_2^{(i)}\,\, C|+|-h_1A_3^{(i)}\,\,-h_2A_1^{(i)}\,\, C|+$$
$$+|-h_1A_3^{(i)}\,\,h_1A_2^{(i)}\,\, C|=h_1(h_1|A_2^{(i)}\,\,A_3^{(i)}\,\, C|-h_2|A_1^{(i)}\,\,A_3^{(i)}\,\, C|+
h_3|A_1^{(i)}\,\,A_2^{(i)}\,\, C|)$$ $$=(-1)^sh_1B_i.$$
similarly we get $M_{i,2}$ and $M_{i,3}.$
\end{proof}

In order to reverse Proposition \ref{cn}, we need to fix some notation. Let $B\in R^{n+1,n}$ be a minimal Hilbert-Burch matrix, such that a row, say the last row, is $H=(h_1,h_2,h_3,0,\ldots,0).$ Then $B$ has the following shape
 \begin{equation}\label{matriceB}
        B=\begin{pmatrix} \\ & A & & & C &\\ \\ h_1 &h_2 & h_3 & 0 &\ldots & 0 \end{pmatrix}. 
 \end{equation}
Moreover we write $B_i$ for the minor obtained from $B$ by removing the $i$-th row, multiplied by $(-1)^i.$ 

\begin{prp}\label{cninv}
With the above notation
let $B\in R^{n+1,n}$ be a minimal Hilbert-Burch matrix, such that the last row is $H=(h_1,h_2,h_3,0,\ldots,0)$ with $(h_1,h_2,h_3)$ a regular sequence.
Let $I$ be the ideal generated by $B_1,\ldots,B_n.$
\par
Then $I$ is a generalized Gorenstein ideal of homological dimension $3$ with maximal $\zeta(I).$

\end{prp}
\begin{proof}
Let us consider the complex
 $$0 \to R \stackrel{\rho}{\rightarrow} R^n\stackrel{\varphi}{\rightarrow}R^n\stackrel{\gamma}{\rightarrow}R$$
where $\rho(1)=(h_1,h_2,h_3,0,\ldots,0),$  $\varphi$ is represented by the matrix $M=(AK|C)$, where $K$ is the matrix of the central map of the Koszul complex on $(h_1,h_2,h_3)$ and $\gamma$ is the map defined by the row $(B_1,\ldots,B_n).$
We have to check that this complex is exact. According to our hypotheses it is useful to rewrite it as follows.
 \begin{equation}\label{mfr0}
   0 \to R \xr{\rho=(\tau,0)} R^3\oplus R^{n-3}\xr{\varphi=\alpha\kappa\oplus\delta}R^n\xr{\gamma}R,
 \end{equation}
where $\tau(1)=(h_1,h_2,h_3)$ and $\alpha,$ $\kappa$ and $\delta$ are the maps represented respectively by $A,$ $K$ and $C.$
\par
Of course $\rho$ is injective and $\Imm\rho\subseteq\Ker\varphi.$
\par
Now we show that $\Ker\varphi\subseteq\Imm\rho.$ Let $u=(u_1,u_2)\in\Ker\varphi,$ $u_1\in R^3,$ $u_2\in R^{n-3}$ i.e. 
$\alpha\kappa(u_1)=0$ and $\delta(u_2)=0.$ Since $B$ is minimal, $\det (A|C)\ne 0,$ so $\alpha$ is injective, therefore 
$\kappa(u_1)=0.$ Consequently 
$u_1\in\Ker\kappa=\Imm\tau.$ Since $B$ is an Hilbert-Burch matrix, $C$ has maximal rank, hence $\delta$ is injective i.e. $u_2=0,$ therefore $u=(u_1,u_2)\in\Imm\rho.$
\par
By Proposition \ref{cn} we have that $\Imm\varphi\subseteq\Ker\gamma,$ so we need to
show that $\Ker\gamma\subseteq\Imm\varphi.$ Let $(v_1,\ldots,v_n)\in\Ker\gamma.$ Then $(v_1,\ldots,v_n,0)$ is a syzygy on 
$(B_1,\ldots,B_n,\det(A|C)).$ So $(v_1,\ldots,v_n,0)$ belongs to the module generated by the columns of $B$ i.e.
 $$\begin{pmatrix}v_1 \\ \ldots \\ v_n \\ 0\end{pmatrix}=\sum_{i=1}^3\lambda_i\begin{pmatrix}b_{1i} \\ \ldots \\ b_{ni} \\ h_i\end{pmatrix}+\sum_{i=4}^n\lambda_i\begin{pmatrix}b_{1i} \\ \ldots \\ b_{ni} \\ 0\end{pmatrix}
\rw$$ $(v_1,\ldots,v_n)=\alpha(\lambda_{1},\lambda_{2},\lambda_{3})+\delta(\lambda),$ where 
$\lambda=(\lambda_4,\ldots,\lambda_n).$
Moreover, since $\lambda_1h_1+\lambda_2h_2+\lambda_3h_3=0,$ we deduce that $(\lambda_1,\lambda_2,\lambda_3)=\kappa(z),$ for some 
$z\in R^3.$ Therefore
 $(v_1,\ldots,v_n)=\alpha\kappa(z)+\delta(\lambda),$ consequently $\varphi(z,\lambda)=(v_1,\ldots,v_n).$
 \end{proof}

By Proposition \ref{cninv} $I_M$ is generated by $n$ among $n+1$ maximal minors of the matrix $B$ in (\ref{matriceB}). 
The next result will allow us to give a structure for $I_M$ in terms of intersection of two simpler ideals.

\begin{lem}\label{reg}
Let $B\in R^{n+1,n},$ with $\rank B=n.$ Let $B_i,$ $1\le i \le n+1$ be the maximal minors of $B.$ Let $I(B)=(B_1,\ldots,B_{n+1}),$ such that $\hgt I(B)=2.$ 
Let $(h_1,\ldots,h_n)$ be the last row in $B.$ Then
 $$I(B)\cap(h_1,\ldots,h_n)=(B_1,\ldots,B_n) \iff B_{n+1} \text{ is regular in }R/(h_1,\ldots,h_n).$$
\end{lem}
\begin{proof}
Take $a_{n+1}\in R,$ such that $a_{n+1}B_{n+1}\in (h_1,\ldots,h_n).$ By assumption there exist $a_1,\ldots,a_n\in R$ such that
$\sum_{i=1}^{n}a_iB_i=a_{n+1}B_{n+1}.$ Since $B$ is an Hilbert-Burch matrix, $(a_1,\ldots,a_{n+1})$ belongs to the $R$-module generated by the columns of $B.$ In particular $a_{n+1}\in (h_1,\ldots,h_n).$
\par
Conversely we have only to prove that $I(B)\cap(h_1,\ldots,h_n)\subseteq(B_1,\ldots,B_n).$ Let $f\in I(B)\cap(h_1,\ldots,h_n),$ so
$f=\sum_{i=1}^{n+1}a_iB_i.$ Since $B_i\in(h_1,\ldots,h_n) $ for $1\le i\le n,$ we get $a_{n+1}B_{n+1}\in(h_1,\ldots,h_n).$ Then, by the assumption, $a_{n+1}\in (h_1,\ldots,h_n),$ i.e. $a_{n+1}=\sum_{i=1}^nu_ih_i.$ For $1\le j\le n,$ $\sum_{i=1}^n b_{ij}B_i=-h_jB_{n+1},$ hence 
$\sum_{i=1}^n u_jb_{ij}B_i=-u_jh_jB_{n+1}.$ Summing up we get $$\sum_{j=1}^n\sum_{i=1}^n u_jb_{ij}B_i=-a_{n+1}B_{n+1},$$ so 
$a_{n+1}B_{n+1}\in (B_1,\ldots,B_n),$ i.e. $f\in (B_1,\ldots,B_n).$
\end{proof}

Using Lemma \ref{reg}, we can give a geometric description of projective schemes having a 
minimal free resolution of type (\ref{mfr0}).

\begin{thm}\label{sch}
Let $X\subset\pp^r,$ $r\ge 3$ be a closed projective scheme, whose defining ideal $I_X$ has a graded minimal free resolution of type (\ref{mfr0}). Let $Z$ be the complete intersection defined by $I(\rho)=I(\tau),$ let $S=V(\det(\alpha\oplus\delta))$ and let $Y$
be the scheme defined by $I(\alpha\oplus\delta,\rho^*).$ If $\codim(S\cap Z)=4$ then $X=Y\cup Z.$
\end{thm}
\begin{proof}
It is enough to observe that since $\codim(S\cap Z)= 4,$ $\det(\alpha\oplus\delta)$ is regular in $R/I_Z.$ So we can apply Theorem \ref{reg} to have our assertion. 
\end{proof}

\begin{rem}
Note that when $\det(\alpha\oplus\delta)$ is a unit, $Y=\emptyset$ and $X=Z.$ When $\det(\alpha\oplus\delta)$ is not a unit, then $X$ is a union of an aCM scheme of codimension $2$ and a complete intersection scheme of codimension $3.$
\end{rem}

\section{The case $n=3$}
\markboth{\it The case $n=3$}{\it The case $n=3$}

Now we will apply the results of previous sections and we will provide an explicit characterization of the graded Betti numbers for generalized Gorenstein ideals having a graded minimal free resolution of the type

 \begin{equation}\label{mfr3} 0 \to F_3\xr{\rho} F_2\xr{\varphi}F_1\xr{\psi}R\to R/I\to 0\end{equation}
with $\rank F_1= \rank F_2 =3$ (and consequently $\rank F_3=1$). 
\par
We start by observing that, by the exactness criterion, $\imm \rho$ is generated by a regular sequence $(h_1,h_2,h_3).$ 
\par
Let $M$ be the matrix associated to $\varphi$ with respect suitable bases.
Since $(h_1,h_2,h_3)$ is a regular sequence, its first syzygy module is generated by the rows of the following matrix
  $$K=\begin{pmatrix}0 & h_3 & -h_2\\ -h_3 &0 & h_1\\ h_2 & -h_1 & 0\end{pmatrix}.$$
As $\varphi\rho=0,$ we get $M=AK,$ where $A\in R^{3,3}.$
Consequently the resolution (\ref{mfr3}) can be written in the following way
 \begin{equation}\label{mfr4} 0 \to F_3\xr{\rho} F_2\xr{\varphi=\alpha\kappa}F_1\xr{\psi}R\to R/I\to 0\end{equation}
where $\alpha$ and $\kappa$ are the maps associated to the mentioned matrices $A$ and $K.$
\begin{prp}\label{ctre}
If
 $$0 \to F_3 \stackrel{\rho}{\rightarrow} F_2\stackrel{\varphi}{\rightarrow}F_1\stackrel{\psi}{\rightarrow}R$$
is a graded minimal free resolution and $M$ is the matrix associated to $\varphi$ with respect suitable bases, then $I_M$ is generated by the maximal minors obtained by deleting one by one the first $3$ rows of the $4\times 3$-matrix
  $$B=\begin{pmatrix} \\ & A & \\ \\ h_1 &h_2 & h_3 \end{pmatrix},$$
where $A$ is the matrix defined above and $(h_1,h_2,h_3)$ generates $\Imm\rho.$
\end{prp}
\begin{proof}
This is a particular case of Proposition \ref{cn}, when $\zeta(I)=0.$
\end{proof}

In order to reverse Proposition \ref{ctre}, we need to fix some notation. Let $B\in R^{4,3}$ be a Hilbert-Burch matrix. Let us consider a row of $B,$ say $H=(h_1,h_2,h_3)$ and let $\widehat{B}$ be the matrix  obtained from $B$ by deleting the row 
$H.$ Let $B_1,B_2,B_3$ be the maximal minors of $B$ including the row $H.$ 

\begin{prp}\label{ctreinv}
With the above notation
let $B\in R^{4,3}$ be a Hilbert-Burch matrix, providing a minimal set of generators, such that one of its rows $H=(h_1,h_2,h_3)$ is a regular sequence.
Let $J$ be the ideal generated by $B_1,B_2,B_3.$
Then a graded minimal free resolution of $R/J$ is
 $$0 \to F_3 \stackrel{\rho}{\rightarrow} F_2\stackrel{\varphi}{\rightarrow}F_1\stackrel{\psi}{\rightarrow}R$$
where $\Imm\rho$ is generated by $(h_1,h_2,h_3),$ $\varphi=\alpha\kappa$ where $\alpha$ is the map associated to the matrix $\widehat{B},$ $\kappa$ is the central map of the Koszul complex on $(h_1,h_2,h_3),$ $\psi$ is the map defined by the row $(B_1,B_2,B_3)$ and $F_1,$ $F_2$ are graded free modules of rank three.
\end{prp}
\begin{proof}
This is a particular case of Proposition \ref{cninv}, when $\zeta(I)=0.$
 \end{proof}

The next proposition describes the graded Betti numbers for an ideal $I\subset R$ whose resolution is of type (\ref{mfr4}).

\begin{prp}\label{bet}
Let $I\subset R$ be a generalized Gorenstein ideal, $\hgt I\ge 2.$
Then there exist six integers $d_1,d_2,d_3,a_1,a_2,a_3,$ with $d_i>0$ and $a_i\ge 0,$ such that the graded minimal free resolution of $R/I$ is
 $$0\to R(-d-a)\to\bigoplus_{i=1}^3R(d_i-d-a)\to\bigoplus_{i=1}^3R(a_i-d_i-a)\to R,$$
where $d=d_1+d_2+d_3$ and $a=a_1+a_2+a_3.$
\par
Conversely if we choose six integers  $d_1,d_2,d_3,a_1,a_2,a_3,$ with $d_i>0$ and $a_i\ge 0,$ then there exists a generalized Gorenstein ideal, $\hgt I\ge 2,$ such that $R/I$ has the the following minimal graded free resolution
 $$0\to R(-d-a)\to\bigoplus_{i=1}^3R(d_i-d-a)\to\bigoplus_{i=1}^3R(a_i-d_i-a)\to R,$$
where $d=d_1+d_2+d_3$ and $a=a_1+a_2+a_3.$
\end{prp}
\begin{proof}
Since $R/I$ has a minimal free resolution of type (\ref{mfr4}), we set $d_1,d_2,d_3,$ the degrees of the complete intersection 
$I(\rho)$ and 
$\alpha:\bigoplus_{i=1}^3R(-e_i)\to \bigoplus_{j=1}^3R(-e'_j).$ We set $a_i=e_i-e'_i,$ for $1\le i\le 3.$ By Proposition \ref{ctre} we see that the degrees of the minimal generators of $I$ are $d_1+a_2+a_3,$ $d_2+a_1+a_3,$ 
$d_3+a_1+a_2$ i.e. $a+d_i-a_i$ for $1\le i\le 3.$ Furthermore since $\varphi=\alpha\kappa,$ a simple computation shows that the shifts of the second module are $e_1-e'_1+d_2+(a+d_1-a_1)=a+d_1+d_2,$ 
$e_2-e'_2+d_3+(a+d_2-a_2)=a+d_2+d_3,$ $e_3-e'_3+d_1+(a+d_3-a_3)=a+d_1+d_3,$  consequently they are $a+d-d_i,$ for $1\le i\le 3.$ Since the map $\rho$ is the map of the complete intersection of type $(d_1,d_2,d_3)$ 
the last graded Betti number is $a+d.$
\par
Conversely let $J=(h_1,h_2,h_3)$ be a complete intersection with $\deg h_i=d_i$ for $1\le i\le 3$ and we choose three forms $g_i,$ 
$\deg g_i=a_i$ for $1\le i\le 3.$  $I=(h_1g_2g_3,h_2g_1g_3,h_3g_1g_2)$ is a required ideal. Namely if we consider the matrix
 $$B=\begin{pmatrix} g_1 & 0 & 0\\ 0 & g_2 & 0\\0 & 0 & g_3 \\ h_1 &h_2 & h_3 \end{pmatrix},$$
it satisfies the hypotheses of Proposition \ref{ctreinv}.
\end{proof}

In order to avoid trivial cases in the sequel we will use the following definition.

\begin{dfn}
A Betti sequence is said to be {\em essential} if it occurs for $R/I$ where $I$ is a homogeneous ideal with $\hgt I\ge 2.$
\end{dfn}

The next theorem will characterize the Betti sequences for generalized Gorenstein ideals of homological dimension $3$ and $\hgt I\ge 2.$

\begin{thm}\label{treB}
A sequence $(a_1,a_2,a_3;b_1,b_2,b_3;s)$ with $a_1\le a_2\le a_3$ and $b_1\ge b_2\ge b_3$ is a essential Betti sequence iff
\begin{itemize}
	\item [1)] $s=\sum_{j=1}^3 b_j-\sum_{i=1}^3 a_i;$
	\item [2)] $\sum_{i=1}^3 a_i<b_2+b_3;$
	\item [3)] $a_j+b_j\le\sum_{i=1}^3 a_i,$ for $1\le j\le 3.$
\end{itemize}
\end{thm}
\begin{proof}
Let $(a_1,a_2,a_3;b_1,b_2,b_3;s)$ be a essential Betti sequence then there is an ideal $I$ of height at least two, such that $R/I$ has the following graded minimal free resolution
 $$0\to R(-s)\to\bo_{j=1}^3 R(-b_j)\to\bo_{i=1}^3 R(-a_i)\to R.$$
Trivially $s=\sum_{j=1}^3 b_j-\sum_{i=1}^3 a_i.$ 
\par
Moreover the last map of the resolution is defined by a regular sequence $h_1,h_2,h_3$ with $\deg h_j=s-b_j,$ for $1\le j\le 3.$ Consequently $s-b_1>0,$ i.e. $\sum_{i=1}^3 a_i<b_2+b_3.$
\par
Let $M=(m_{ij})$ be a matrix associated to the central map of the above resolution, with $\deg m_{ij}=b_j-a_i.$ We have that
$\sum_{j=1}^3 m_{ij}h_j=0,$ for $1\le i\le 3.$ Consequently 
 $$(m_{i1},m_{i2},m_{i3})\in ((0,h_3,-h_2),(-h_3,0,-h_2),(h_2,-h_1,0)),\,\,1\le i\le 3.$$
Therefore $\deg m_{ij}\ge\deg h_k,$ with $i\ne j\ne k\ne i,$ i.e.
 $$b_j-a_i\ge s-b_k\rw b_i\le a_j+a_k\rw a_j+b_j\le\sum_{i=1}^3 a_i, \text{ for } 1\le j\le 3.$$
Conversely let us suppose that the sequence $(a_1,a_2,a_3;b_1,b_2,b_3;s)$ satisfies the conditions $1,2,3$ above. 
We set $c_j=s-b_j,$ for $1\le j\le 3.$ Then $c_3\ge c_2\ge c_1$ and $c_1>0$ by the assumption $2.$ Hence $c_j>0$ for 
$1\le j\le 3.$ Now we set $t_j=\sum_{i=1}^3 a_i-a_j-b_j.$ By the assumption $3,$ $t_j\ge 0,$ for $1\le j\le 3.$
Note that $\sum_{j=1}^3 t_j+\sum_{j=1}^3 c_j=s,$ $\sum_{i=1}^3 t_i+\sum_{i=1}^3 c_i-c_j=b_j$ and 
$\sum_{i=1}^3 t_i+c_j-t_j=a_j.$ Now, applying Proposition \ref{bet} to the integers $c_1,c_2,c_3,t_1,t_2,t_3$ we get that 
$(a_1,a_2,a_3;b_1,b_2,b_3;s)$ is an essential Betti sequence.
\end{proof}


\section{Graded Betti numbers for ideals $I_M$}
\markboth{\it Graded Betti numbers for ideals $I_M$}{\it Graded Betti numbers for ideals $I_M$}

In this section we study the graded Betti numbers for generalized Gorenstein ideals $I_M,$ arising from a minimal presentation matrix $M=(m_{ij}).$
A graded minimal resolution for such ideals can be written in the following way
 \begin{equation}\label{gr}0\to R(-s)\to\bo_{j=1}^nR(-b_j)\to\bo_{i=1}^nR(-a_i)\to R\to R/I_M\to 0,\end{equation}
where $a_1\le\ldots\le a_n,$ $b_1\ge\ldots\ge b_n$ and $s=\sum_{j=1}^nb_j-\sum_{i=1}^na_i.$ 
We will set also $c_j=s-b_j,$ for $1\le j\le n.$ Note that $C=(c_1\ldots c_n)^T$ is the degree vector of the leftmost map of the resolution.
It is easy to check that $a_i<b_{n+1-i}$ for $1\le i\le n,$ and $a_2<b_n,$ $a_3<b_{n-1}.$ Moreover $b_{n-2}<s\le a_1+a_2+a_3.$
Now we set $d_{ij}=\deg m_{ij}.$ Note that $d_{ij}=b_j-a_i,$ so $d_{ij}\ge d_{i+1j}$ and $d_{ij}\ge d_{ij+1}$ for $1\le i\le n-1$ and $1\le j\le n-1.$ So if $d_{hk}=0$ for some $(h,k)$ then $m_{ij}=0$ for every $i\ge h$ and $j\ge k.$
The matrix $D=(d_{ij})\in\zz^{n,n}$ is called the degree matrix of $M.$ The degree matrix of $M$ does not determine, in general, the graded Betti numbers of $I_M.$

\begin{exm}
Let us consider the ideals
 $$I=(xyz,yzt,ztu,tuv,uvx,vxy),\, J=(xyzt,yztu,ztuv,tuvx,uvxy,vxyz).$$
Their graded minimal free resolutions are
 $$0\to R(-6)\to R(-4)^6\to R(-3)^6\to R\to R/I\to 0$$
and
 $$0\to R(-6)\to R(-5)^6\to R(-4)^6\to R\to R/J\to 0.$$
\end{exm}

However if we know one of the $c_j$'s in addition to the degree matrix, then the graded Betti numbers are determined.

\begin{prp}
Let $M$ be a minimal presentation matrix and let $D=(d_{ij})$ be the degree matrix of $M.$ Let $c_r$ be the degree of the $r$-th component of $C.$ Then the graded Betti numbers of $R/I_M$ are
$s=\sum_{i=1}^n d_{ii};$ $b_j=s+d_{rj}-d_{rr}-c_r,$ for $1\le j \le n;$ $a_i=s-d_{ir}-c_r,$ for $1\le i \le n.$
\end{prp}
\begin{proof}
By the exactness of (\ref{gr}) $s=\sum_{i=1}^n(b_i-a_i)=\sum_{i=1}^n d_{ii}.$
\par
Furthermore $$a_i=a_i+s-b_r-(s-b_r)=s-(b_r-a_i)-c_r=s-d_{ri}-c_r.$$
 $$b_j=a_r+d_{rj}=s-d_{rr}-c_r+d_{rj}.$$
\end{proof}

\begin{prp}\label{bettiinc}
Let $$(a_1,\ldots,a_n;b_1,\ldots,b_n;s)$$ be an essential Betti sequence. Let $u_i\ge 0,$ $1\le i\le n$ be any integers. 
We set $u=\sum_{i=1}^nu_i.$
Then the sequence
 $$(a_1+u-u_1,\ldots,a_n+u-u_n;b_1+u,\ldots,b_n+u;s+u)$$
is an essential Betti sequence.
\end{prp}
\begin{proof}
By the assumptions there exists an ideal $I\subset R=k[x_1,\ldots,x_r],$ $\hgt I \ge 2,$ having a resolution of the type
 $$ 0\to R(-s)\to\bo_{j=1}^n R(-b_j)\xr{\varphi}\bo_{i=1}^n R(-a_i)\to R\to R/I\to 0,$$
where $\varphi(e_i)=(m_{1i},\ldots,m_{ni})$ and $(e_1,\ldots,e_n)$ is a basis of $\bo_{j=1}^n R(-b_j).$
Let $S=R[y_1,\ldots,y_n],$ where the $y_i$'s are new variables. 
Let $$\varphi':\bo_{j=1}^n S(-b_j-u)\to\bo_{i=1}^n S(-a_i-u+u_i)$$ be the map defined by 
$$\varphi'(e_i')=(m_{1i}y_1^{u_1},\ldots,m_{ni}y_n^{u_n}).$$
By Theorem \ref{pres} one sees that the matrix $M'=(m_{ij}y_j^{u_j})$ (matrix associated to $\varphi'$) is a minimal presentation matrix, so it defines an ideal $I_{M'},$ whose minimal free resolution looks like
 $$ 0\to S(-s-u)\to\bo_{j=1}^n S(-b_j-u)\xr{\varphi'}\bo_{i=1}^n S(-a_i-u+u_i)\to S\to S/I_{M'}\to 0.$$
Moreover if $I_M=(g_1,\ldots,g_n),$ then $I_{M'}=(g_1',\ldots,g_n'),$ where 
 $$g_i'=g_i\prod_{\stackrel {j=1}{j\ne i}}^ny_j^{u_j}.$$
\end{proof}

\begin{dfn}
We will say that a Betti sequence $(a_1,\ldots,a_n;b_1,\ldots,b_n;s)$ is {\em minimal} if the $n$ sequences
$(a_1,a_2-1\ldots,a_n-1;b_1-1,\ldots,b_n-1;s-1),$ $\ldots,$ $(a_1-1,a_2-1\ldots,a_n;b_1-1,\ldots,b_n-1;s-1)$
are not Betti sequences.
\end{dfn}

Of course, by Proposition \ref{bettiinc}, it is enough to find the minimal Betti sequences for determine all Betti sequences for ideals $I_M.$
\par
In order to give more information about Betti sequences for ideals $I_M,$ we give the following definition, which arises from perfect ideals of height $2$ (see \cite{G}).

\begin{dfn}
We will say that a sequence $(a_1 \le\ldots\le a_n;b_1 \ge\ldots\ge b_n;s)$ is a Gaeta sequence if 
$s=\sum_{j=1}^n b_j-\sum_{i=1}^n a_i$ and $b_{n+2-i}>a_i$ for $2\le i\le n.$
\end{dfn}

\begin{rem}
Of course not every Gaeta sequence is an essential Betti sequence. For instance the sequence
  $$(3,3,3,3;5,5,5,5;8)$$
is a Gaeta sequence. If it were an essential Betti sequence there should be an ideal $I\subset R=k[x,y,z],$ of height at least 
$2,$ having this Betti sequence.  But $H_{R/I}(6)=0,$ so $R/I$ is an Artinian algebra, therefore $R/I$ is a Gorenstein algebra of codimension $3,$ which is a contradiction because $I$ has an even number of minimal generators.
\end{rem}

Our aim is to understand when a Gaeta sequence is an essential Betti sequence.
\par
The next result will permit us to reduce the study of the essential Betti sequences to the Gaeta sequences.

\begin{thm}\label{noGaeta}
Let $(a_1 \le\ldots\le a_n;b_1 \ge\ldots\ge b_n;s)$ be a sequence such that $b_{n+2-t}\le a_t$ for some $t\ge 4.$ We set
$d=\sum_{i=t}^n(b_{n+1-i}-a_i).$ It is an essential Betti sequence iff 
\begin{itemize}
\item [1)] $s=\sum_{j=1}^n b_j-\sum_{i=1}^n a_i;$
\item [2)] $(a_1-d,\ldots ,a_{t-1}-d;b_{n+2-t}-d,\ldots b_n-d;s-d)$ is an essential Betti sequence;
\item [3)] $b_{n+1-i}>a_i$ for $i\ge t.$
\end{itemize}
\end{thm}
\begin{proof}
Let us suppose that $(a_1 \le\ldots\le a_n;b_1 \ge\ldots\ge b_n;s)$ is an essential Betti sequence, so there is an ideal $I,$ 
$\hgt I\ge 2,$ such that the graded minimal free resolution of $R/I$ looks like 
 $$0\to R(-s)\to\bo_{j=1}^n R(-b_j)\to\bo_{i=1}^n R(-a_i)\to R.$$
The conditions $1$ and $3$ are well known for general facts.
Let $M=(m_{ij})$ be a matrix associated to the central map of the resolution, such that $\deg m_{ij}=b_j-a_i.$
Let $(g_1,\ldots,g_n)=\gamma(M).$
Since $b_{n+2-t}\le a_t,$ $m_{ij}=0$ for $i\ge t$ and $j\ge n+2-t.$ Let $M'=(m_{ij})$ be the square submatrix of $M$ where $1\le i\le t-1$ and $n+2-t\le j\le n$ and $D=(m_{ij})$ where $t\le i\le n$ and $1\le j\le n+1-t.$ 
Note that $\rank M'\le t-2$ since  $(g_1 \ldots g_{t-1})M'=0.$ Furthermore $\rank M'\ge t-2$ since $\rank M=n-1.$ So $\rank M'=t-2.$
Moreover $\det D\ne 0.$ Indeed, because of the vanishing of the maximal minors of the submatrix of $M$ obtained by removing a column $C_j$ with $1\le j\le n+1-t,$ there is a column $C_k$ with $n+2-t\le k\le n$ such that the maximal minors of the submatrix of $M$ obtained by removing $C_k$ are not vanishing multiple of $g_1,\ldots,g_n.$ Since such minors are multiple of $\det D$ we get that $\det D\ne 0.$ Consequently $\gamma(M')=(\det D)^{-1}(g_1,\ldots,g_{t-1}).$ Now since every syzygy on $\gamma(M')$ is also a syzygy on 
$g_1,\ldots,g_n$ and since $\det D\ne 0$ this syzygy must be in the span of $C_j$ for $n+2-t\le j\le n.$ So by 
Proposition \ref{carm}, $M'$ is a presentation matrix and the Betti sequence of $I_{M'}$ is 
$(a_1-d,\ldots ,a_{t-1}-d;b_{n+2-t}-d,\ldots b_n-d;s-d).$
\par
Conversely we suppose that the conditions $1,$ $2$ and $3$ are satisfied. In particular, by condition $2$ there exists a presentation matrix $M'=(m'_{ij}),$ with $\deg m'_{ij}=(b_{n+1-t-j}-d)-(a_{i}-d)=b_{n+1-t-j}-a_{i}$ of size $t-1$ such that $I_{M'}$ has Betti sequence $(a_1-d,\ldots ,a_{t-1}-d;b_{n+2-t}-d,\ldots b_n-d;s-d).$ Now we define a square matrix $M=(m_{ij}),$ of size $n,$ in the following way
 $$m_{ij}=\left\{
\begin{array}{ll}
m'_{i,j-(n+1-t)} & \text{for } 1\le i\le t-1,\, n+2-t\le j\le n \\
y_i^{b_j-a_i} & \text{for } i+j=n,\, t-1\le i\le n-1 \\
z_i^{b_j-a_i} & \text{for } i+j=n+1,\, t\le i\le n \\
0 & \text{elsewhere}
\end{array} \right. ,
 $$
where $y_j$ and $z_j$ are new variables for every $j.$ The condition $3$ guarantees that the exponents of $y_j$ and $z_j$ are positive integers. Since $\rank M'=t-2,$ we have $\rank M=n-1.$ We set $(g'_1,\ldots,g'_{t-1})=\gamma(M').$ Now if we set 
$(g_1,\ldots,g_n)=\gamma(M),$ we see that $g_i=g'_i\prod_{i=t}^n z_i^{b_{n-i}-a_i}$ for $1\le i\le t-1$ and 
$g_i=g'_{t-1}\prod_{h=t-1}^{i-1}y_h^{b_{n-h}-a_h}\prod_{h=i+1}^{n}z_h^{b_{n+1-h}-a_h}$ for $t\le i\le n.$
Note that $\deg g_i=a_i$ for $1\le i\le n.$ In order to show that $M$ is a minimal presentation matrix, we will use 
Theorem \ref{pres}. We set $(h'_1,\ldots,h'_{t-1})=\gamma(M'^T).$ By Theorem \ref{pres}, the ideal $J'$ generated by the components of $\gamma(M'^T)$ has $\depth J'\ge 3.$ Since $\gamma(M^T)=(0,\ldots,0,h'_1,\ldots,h'_{t-1}),$ the ideal $J$ generated by the components of $\gamma(M^T)$ coincides with $J',$ so it has depth greater than or equal to $3$ too. By Lemma \ref{lcof}, 
$M^C=u\gamma(M)^T\gamma(M^T).$ So we need only to show that $u$ is a unit. To do this we compute the cofactor $M_{1n}$ of the entry in position $(1,n).$ 
 $$M_{1n}=(-1)^{n+1}g'_1h'_{t-1}\prod_{i=t}^n z_i^{b_{n-i}-a_i}=(-1)^{n+1}g_1h'_{t-1};$$
since $h'_{t-1}$ is the $n$-th component of $\gamma(M^T),$ we are done.
\end{proof}

\begin{cor}\label{t4}
Let $(a_1 \le\ldots\le a_n;b_1 \ge\ldots\ge b_n;s)$ be a sequence such that $b_{n-2}\le a_4.$  We set
$d=\sum_{i=4}^n(b_{n+1-i}-a_i).$ It is an essential Betti sequence iff 
\begin{itemize}
\item [1)] $s=\sum_{j=1}^n b_j-\sum_{i=1}^n a_i;$
\item [2)] $a_j+b_{n-3+j}+d\le a_1+a_2+a_3< b_{n-1}+b_n+d,$ for $j=1,2,3.$
\item [3)] $b_{n+1-i}>a_i$ for $i\ge 4.$
\end{itemize}
\end{cor}
\begin{proof}
According to Theorem \ref{noGaeta}, we need to show that $(a_1-d,a_2-d,a_3-d;b_{n-2}-d,b_{n-1}-d,b_{n}-d;s-d)$ is a essential Betti sequence. Now it is enough to use Theorem \ref{treB} to verify this fact.
\end{proof}

\begin{rem}\label{remga}
Note that by iterating the procedure of Theorem \ref{noGaeta} any sequence $\beta=(a_1 \le\ldots\le a_n;b_1 \ge\ldots\ge b_n;s)$
can be transformed in a Gaeta sequence $\beta'=(a'_1 \le\ldots\le a'_m;b'_1 \ge\ldots\ge b'_m;s').$ 
\end{rem}

\begin{cor}
Let $\beta=(a_1 \le\ldots\le a_n;b_1 \ge\ldots\ge b_n;s)$ be a sequence. Using the same notation of Remark \ref{remga},
$\beta$ is an essential Betti sequence iff the Gaeta sequence $\beta'$ is an essential Betti sequence.
\end{cor}
\begin{proof}
Taking into account Remark \ref{remga}, it is an easy application of Theorem \ref{noGaeta}.
\end{proof}

Now we study the essential Betti sequences of the type
  $$(a,\ldots,a;b,\ldots,b;s).$$
	
\begin{dfn}
Let $M$ be a square matrix of size $n.$ The matrix $M=(m_{ij})$ is said to be {\em bidiagonal} iff $m_{ij}=n_{ij}=0$ for $j\ne i,i+1,$ $1\le i\le n$ 
(here $m_{n,n+1}$ means $m_{n1}$).
\end{dfn}

\begin{lem}\label{prodmat}
Let $S=k[x_1,\ldots,x_r,y_1,\ldots,y_s].$ Let $M=(m_{ij})$ and $N=(n_{ij})$ be two minimal presentation bidiagonal square matrices of size $n$ such that 
$m_{ij}$ are forms of degree $d$ in $k[x_1,\ldots,x_r],$ $n_{ij}$ are forms of degree $e$ in $k[y_1,\ldots,y_s].$ 
Let $M*N=(t_{ij})$ be the matrix such that $t_{ii}=m_{ii}n_{ii}$ and $t_{i,i+1}=-m_{i,i+1}n_{ii+1}$ and $t_{ij}=0$ otherwise.
\par
Then $M*N$ is a presentation bidiagonal matrix. Moreover if $\gamma(M^T)=(h_1,\ldots,h_n)$ and $\gamma(N^T)=(k_1,\ldots,k_n)$ then 
$\gamma((M*N)^T)=(h_1k_1,\ldots,h_nk_n).$
\end{lem}
\begin{proof}
Note that 
 $$\det (M*N)=\prod_{i=1}^n t_{ii}+(-1)^{n+1}\prod_{i=1}^nt_{i,i+1}=$$
$$=\prod_{i=1}^n m_{ii}n_{ii}+(-1)^{n+1}(-1)^n\prod_{i=1}^n m_{i,i+1}n_{i,i+1}=$$
$$=\prod_{i=1}^n m_{ii}n_{ii}-\prod_{i=1}^n m_{i,i+1}n_{i,i+1}=0.$$
To show that $M*N$ is a minimal presentation matrix we use Theorem \ref{pres}. At first we need to compute the cofactors of $M*N.$ Such a computation can be found, for instance, in the paper \cite{RZ3} on page 281. From this computation follows immediately that  $\det (M*N)_{ij}=\det M_{ij}\det N_{ij}$ (where with the index $ij$ we mean the submatrix obtained by removing $i$-th row and $j$-th column). Consequently we get that $\gamma((M*N)^T)=(h_1k_1,\ldots,h_nk_n).$ Since $\gamma((M)^T)$ and $\gamma((N)^T)$ consisting of forms living in polynomial rings in different variables we deduce that the ideal generated by $\gamma((M*N)^T)$ has depth at least $3.$ From the same computation follows also that 
$(M*N)^C=\gamma((M*N))^T\gamma((M*N)^T).$ 
\end{proof}

\begin{prp}\label{omognec}
If a sequence $(a_1,\ldots,a_n;b_1,\ldots,b_n;s)$ of positive integers, with $a_1=\ldots=a_n=a,$ $b_1=\ldots=b_n=b,$ is a essential Betti sequence then
\begin{itemize}
\item [1)] $s=n(b-a);$
\item [2)] $na<(n-1)b \le (n+1)a;$ moreover, when $n$ is even, $(n-1)b < (n+1)a.$ 
\end{itemize}
\end{prp}
\begin{proof}
Let us suppose that $(a_1,\ldots,a_n;b_1,\ldots,b_n;s)$ with $a_1=\ldots=a_n=a,$ $b_1=\ldots=b_n=b,$ is a essential Betti sequence. Then there exists an ideal $I,$ $\hgt I\ge 2,$ whose resolution is
 $$0\to R(-s)\to R(-b)^n\to R(-a)^n \to R\to R/I\to 0.$$
The condition $1$ is trivial. Moreover, since $s>b,$ $n(b-a)>b,$ so $na<(n-1)b.$ \par
Since $\depth(R/I)=\depth(R)-3$ we can reduce to a ring in only $3$ variables. So we can suppose that $R=k[x_1,x_2,x_3].$ Of course we have that $H_{R/I}(s-2)\ge 0.$ Therefore
 $$0\le H_{R/I}(s-2)=\binom{s}{2}-n\binom{s-a}{2}+n\binom{s-b}{2}-\binom{0}{2}=$$
 $$=\frac{1}{2}n(b-a)[a(n+1)-b(n-1)]$$
that implies $(n-1)b\le (n+1)a.$ Moreover, when $n$ is even, since $R/I$ cannot be Gorenstein, hence it cannot be Artinian, so $H_{R/I}(s-2)>0,$ so for $n$ even we have $(n-1)b < (n+1)a.$
\end{proof}

\begin{thm}\label{omogdisp}
A sequence $(a_1,\ldots,a_n;b_1,\ldots,b_n;s)$ of positive integers, with $a_1=\ldots=a_n=a,$ $b_1=\ldots=b_n=b,$ $n$ odd is an essential Betti sequence iff
\begin{itemize}
\item [1)] $s=n(b-a);$
\item [2)] $na<(n-1)b \le (n+1)a.$
\end{itemize}
\end{thm}
\begin{proof}
The condition is necessary by Proposition \ref{omognec}.
\par
Conversely let $(a_1,\ldots,a_n;b_1,\ldots,b_n;s)$ with $a_1=\ldots=a_n=a,$ $b_1=\ldots=b_n=b,$ $n$ odd a sequence satisfying the conditions 
$1)$ and $2).$  
\par
By subtracting $(n-1)a,$ the condition $2)$ becomes $$\frac{n-1}{2}(b-a) \le  a \le (n-1)(b-a)-1.$$ 
Now we work by induction on $b-a.$ For $b-a=1$ our condition becomes $s=n$ and $\frac{n-1}{2}\le  a \le n-2.$ Using Corollary 2.11 in \cite{RZ3} we can produce an ideal $I,$ $\hgt I\ge 2,$ in $R$ such that $R/I$ has the requested Betti sequence.  Let us suppose that we have realized an algebra $R/I$ having the requested Betti sequence when $b-a=h$. We need to construct algebras $R/I$ with Betti sequence satisfying $\frac{(n-1)}{2}(h+1)\le a \le (n-1)(h+1)-1$ and $s=n(h+1).$ Using Proposition \ref{bettiinc} for $u_i=1$ for every $i,$ we realize the Betti sequences satisfying $s=n(h+1)$ and 
$\frac{(n-1)(h+2)}{2}\le a \le (n-1)(h+1)-1.$ So, it remains to build the Betti sequences such that 
$$s=n(h+1) \text{ and } \frac{(n-1)}{2}(h+1)\le a \le \frac{(n-1)}{2}(h+2)-1$$ 
i.e. $\frac{n-1}{2}h+1\le s-b \le\frac{n-1}{2}(h+1).$ We are interested on the integer $s-b$ since it is the degree of the components of vector $\gamma(M^T),$ where $M$ is the presentation matrix which we will use to realize these Betti sequences. Note that $h$ and $s-b$ determine all the Betti sequence. For $h=1$ we have realized every Betti sequence such that $1\le s-b \le \frac{n-1}{2},$ using bidiagonal matrices. Moreover, by the inductive hypothesis, we have also realized every Betti sequence such that $b-a=h$ and $1\le s-b \le h\frac{n-1}{2},$ using again bidiagonal matrices. Now let $1\le t \le \frac{n-1}{2}$ and let $M$ be a minimal presentation bidiagonal matrix realizing the Betti sequence such that $s=n,$ $b-a=1$ and $s-b=t.$ Let $N$ be a minimal presentation bidiagonal matrix realizing the Betti sequence such that $s=nh,$ $b-a=h$ and $s-b=h\frac{n-1}{2}.$ Applying Lemma \ref{prodmat} we get a matrix $M*N$ realizing the Betti sequence such that 
$s=n(h+1),$ $b-a=h+1$ and $s-b=h\frac{n-1}{2}+t.$
\end{proof}

\begin{prp}\label{omogparisuff}
A sequence $(a_1,\ldots,a_n;b_1,\ldots,b_n;s)$ of positive integers, with $a_1=\ldots=a_n=a,$ $b_1=\ldots=b_n=b,$ $n$ even is an essential Betti sequence provided that 
\begin{itemize}
\item [1)] $s=n(b-a);$
\item [2)] $na<(n-1)b \le na+\frac{n-2}{2}(b-a).$
\end{itemize}
\end{prp}
\begin{proof}
At first we observe that the condition $2$ is equivalent to 
 $$\frac{n}{2}(b-a)\le a\le (n-1)(b-a)-1 \iff  1\le s-b \le \frac{n-2}{2}(b-a).$$
We proceed analogously to the proof of Theorem \ref{omogdisp}.
Now we work by induction on $b-a.$ For $b-a=1$ our conditions become $s=n$ and $\frac{n-2}{2}\le  a \le n-2.$ Using Corollary 2.11 in \cite{RZ3} we can produce an ideal $I$ in $R,$ $\hgt I\ge 2,$ such that $R/I$ has the requested Betti sequence.
\par
Let us suppose that we have realized an algebra $R/I$ having the requested Betti sequence when $b-a=h$. We need to construct algebras $R/I$ with Betti sequence satisfying $\frac{n}{2}(h+1)\le a \le (n-1)(h+1)-1$ and $s=n(h+1).$ Using Proposition \ref{bettiinc} for $u_i=1$ for every $i,$ we realize the Betti sequences satisfying $s=n(h+1)$ and 
$\frac{n}{2}h+n-1\le a \le (n-1)(h+1)-1.$ So, it remains to build the Betti sequences such that 
    $$s=n(h+1) \text{ and } \frac{n}{2}(h+1)\le a \le \frac{n}{2}h+n-2$$ 
 i.e. $\frac{n-2}{2}h+1\le s-b \le\frac{n-2}{2}(h+1).$ 
We are interested on the integer $s-b$ since it is the degree of the components of vector $\gamma(M^T),$ where $M$ is the presentation matrix which we will use to realize these Betti sequences. Note that $h$ and $s-b$ determine all the Betti sequence. For $h=1$ we have realized every Betti sequence such that $1\le s-b \le \frac{n-2}{2},$ using bidiagonal matrices. Moreover, by the inductive hypothesis, we have also realized every Betti sequence such that $b-a=h$ and $1\le s-b \le h\frac{n-2}{2},$ using again bidiagonal matrices. Now let $1\le t \le \frac{n-2}{2}$ and let $M$ be a minimal presentation bidiagonal matrix realizing the Betti sequence such that $s=n,$ $b-a=1$ and $s-b=t.$ Let $N$ be a minimal presentation bidiagonal matrix realizing the Betti sequence such that $s=nh,$ $b-a=h$ and $s-b=h\frac{n-2}{2}.$ Applying Lemma \ref{prodmat} we get a matrix $M*N$ realizing the Betti sequence such that 
$s=n(h+1),$ $b-a=h+1$ and $s-b=h\frac{n-2}{2}+t.$
\end{proof}

\begin{rem}
Unfortunately our construction does not allow building all the sequences satisfying the conditions of Proposition \ref{omognec}. For instance the sequence $(5,5,5,5;8,8,8,8;12)$ cannot be build with the tools of Proposition \ref{omogparisuff}. Nevertheless it is an essential Betti sequence. In fact, using Macaulay 2, one can verify that the ideal $I=(f_1,f_2,f_3,f_4)$ with
$$
 \begin{array}{l}
 f_1=-{x}_{3} {y}_{4} {y}_{5} {z}_{4} {z}_{5}-{y}_{1} {y}_{4} {y}_{5} {z}_{4} {z}_{6}+{x}_{3} {y}_{4} {y}_{5} {z}_{1} {z}_{8}+{y}_{1} {y}_{4} {y}_{5} {z}_{2} {z}_{8} \\
f_2={x}_{1} {x}_{2} {x}_{3} {z}_{4} {z}_{5}+{x}_{1} {x}_{2} {y}_{1} {z}_{4} {z}_{6}+{y}_{1} {y}_{2} {y}_{3} {z}_{4} {z}_{7}+ \\ 
\hspace{4cm} -{x}_{1} {x}_{2} {x}_{3} {z}_{1} {z}_{8} -{x}_{1} {x}_{2} {y}_{1} {z}_{2} {z}_{8}-{y}_{1} {y}_{2} {y}_{3} {z}_{3} {z}_{8} \\
f_3={x}_{3} {y}_{2} {y}_{3} {z}_{3} {z}_{5}+{y}_{1} {y}_{2} {y}_{3} {z}_{3} {z}_{6}-{x}_{3} {y}_{2} {y}_{3} {z}_{1} {z}_{7}-{y}_{1} {y}_{2} {y}_{3} {z}_{2} {z}_{7} \\
f_4=-{x}_{1} {x}_{2} {x}_{3} {z}_{3} {z}_{5}-{x}_{1} {x}_{2} {y}_{1} {z}_{3} {z}_{6}+{x}_{1} {x}_{2} {x}_{3} {z}_{1} {z}_{7}+ \\
\hspace{4cm} +{x}_{1} {x}_{2} {y}_{1} {z}_{2} {z}_{7}+{x}_{3} {y}_{4} {y}_{5} {z}_{4} {z}_{7}-{x}_{3} {y}_{4} {y}_{5} {z}_{3} {z}_{8}
  \end{array}
$$
has height $2$ and the above Betti sequence.
\end{rem}

\vspace{1cm}

{\f
{\sc (A. Ragusa) Dip. di Matematica e Informatica, Universit\`a di Catania,\\
                  Viale A. Doria 6, 95125 Catania, Italy}\par
{\it E-mail address: }{\tt ragusa@dmi.unict.it} \par
{\it Fax number: }{\f +39095330094} \par
\vspace{.3cm}
{\sc (G. Zappal\`a) Dip. di Matematica e Informatica, Universit\`a di Catania,\\
                  Viale A. Doria 6, 95125 Catania, Italy}\par
{\it E-mail address: }{\tt zappalag@dmi.unict.it} \par
{\it Fax number: }{\f +39095330094}
}

\end{document}